\documentclass[11pt]{amsart}
\usepackage{amssymb}
\usepackage{amsthm}
\usepackage{amsmath}
\usepackage{latexsym}
\usepackage{comment}
\usepackage{hyperref}
\usepackage{verbatim}
\usepackage{xypic}
\usepackage{graphicx}
\usepackage[utf8]{inputenc}
\usepackage[margin=1in]{geometry}
\newtheorem{thm}{Theorem}[section]
\newtheorem{prop}[thm]{Proposition}

\newtheorem{lemma}[thm]{Lemma}
\newtheorem{cor}[thm]{Corollary}

\theoremstyle{definition}

\theoremstyle{definition} \newtheorem{rmk}[thm]{Remark}
\newcommand{\cc}{\mathbb{C}}

\newcommand{\qq}{\mathbb{Q}}
\newcommand{\zz}{\mathbb{Z}}
\newcommand{\ff}{\mathbb{F}}

\newcommand{\proj}{\mathbb{P}}

\newcommand{\Gal}{\mathrm{Gal}}

\newcommand{\SL}{\mathrm{SL}}

\newcommand{\Sp}{\mathrm{Sp}}
\newcommand{\GSp}{\mathrm{GSp}}

\newcommand{\Aut}{\mathrm{Aut}}
\newcommand{\Spec}{\mathrm{Spec}}

\newcommand{\ab}{\mathrm{ab}}

\newcommand{\unr}{\mathrm{unr}}
\newcommand{\tp}{\mathrm{top}}

\graphicspath{{pics/}}

\title{An abelian subextension of the dyadic division field of a hyperelliptic Jacobian}
\author{Jeffrey Yelton}
\newcommand{\acr}{\newline\indent}
\address{\llap{*\,}University of Milan \acr
Department of Mathematics \acr
Via Cesare Saldini, 50 \acr
Milan \acr
ITALY}


\begin{document}

\maketitle

\begin{abstract}

Given a field $k$ of characteristic different from $2$ and an integer $d \geq 3$, let $J$ be the Jacobian of the ``generic" hyperelliptic curve given by $y^2 = \prod_{i = 1}^d (x - \alpha_i)$, where the $\alpha_i$'s are transcendental and independent over $k$; it is defined over the transcendental extension $K / k$ generated by the symmetric functions of the $\alpha_i$'s.  We investigate certain subfields of the field $K_{\infty}$ obtained by adjoining all points of $2$-power order of $J(\bar{K})$.  In particular, we explicitly describe the maximal abelian subextension of $K_{\infty} / K(J[2])$ and show that it is contained in $K(J[8])$ (resp. $K(J[16])$) if $g \geq 2$ (resp. if $g = 1$).  On the way we obtain an explicit description of the abelian subextension $K(J[4])$, and we describe the action of a particular automorphism in $\Gal(K_{\infty} / K)$ on these subfields.

\end{abstract}

\section{Introduction} \label{sec1}

Let $k$ be any field of characteristic different from $2$; let $\alpha_1, ... , \alpha_d$ be transcendental and independent over $k$ for some integer $d \geq 3$; and let $K$ denote the extension of $k$ obtained by adjoining the symmetric functions of the $\alpha_i$'s with separable closure denoted $\bar{K}$.  The equation given by 
\begin{equation} \label{eq hyperelliptic}
y^2 = \prod_{i = 1}^{d} (x - \alpha_i)
\end{equation}
defines a hyperelliptic curve $C$ of genus $g := \lfloor (d - 1) / 2 \rfloor$ over $K$.  Its Jacobian, denoted by $J$, is a principally polarized abelian variety over $K$ of dimension $g$.  For each integer $n \geq 1$, we write $J[2^n] \subset J(\bar{K})$ for the $2^n$-torsion subgroup of $J$ and $K_n := K(J[2^n])$ for the (finite algebraic) extension of $K$ obtained by adjoining the coordinates of the points in $J[2^n]$ to $K$; we denote the (infinite algebraic) extension $\bigcup_{n = 1}^{\infty} K_n$ by $K_{\infty} = K(J[2^{\infty}])$.  Let $T_2(J)$ denote the $2$-adic Tate module of $J$; it is a free $\zz_2$-module of rank $2g$ given by the inverse limit of rank-$2g$ $\zz / 2^n \zz$-modules $J[2^n]$ with respect to the multiplication-by-$2$ map.  The canonical principal polarization on $J$ defines the Weil pairing $e_2 : T_2(J) \times T_2(J) \to \zz_2$; it is a nondegenerate, skew-symmetric, $\zz_2$-bilinear pairing on $T_2(J)$.

We have the natural action of the absolute Galois group $G_K = \Gal(\bar{K} / K)$ on each $J[2^{n}]$.  It is well known that this action respects the Weil pairing $e_2$ up to multiplication by the cyclotomic character $\chi_2 : G_K \to \zz_2^{\times}$ (in particular, this implies that $K_{\infty}$ contains the multiplicative subgroup $\mu_2$ of $2$-power roots of unity in the separable closure of $k$).  Each element $\sigma \in G_K$ therefore acts as an automorphism in the group 
$$\GSp(T_2(J)) := \{\sigma \in \mathrm{Aut}_{\zz_{2}}(T_{2}(J))\ |\ e_{2}(P^{\sigma}, Q^{\sigma}) = \chi_{2}(\sigma) e_{2}(P, Q) \ \forall P, Q \in T_{2}(J)\}$$
 of symplectic similitudes.  We denote this natural Galois action by $\rho_2 : G_K \to \GSp(T_2(J))$ and each modulo-$2^n$ action by $\bar{\rho}_{2^{n}} : G_{K} \to \GSp(J[2^{n}])$.  For any field $F$, we write $F(\mu_2)$ for the algebraic extension of $F$ obtained by adjoining all $2$-power roots of unity.  Clearly the image of $\rho_2$ is contained in the symplectic group 
$$\Sp(T_2(J)) := \{\sigma \in \mathrm{Aut}_{\zz_{2}}(T_{2}(J))\ |\ e_{2}(P^{\sigma}, Q^{\sigma}) = e_{2}(P, Q) \ \forall P, Q \in T_{2}(J)\}$$
 if and only if $K = K(\mu_2)$.  For each $n \geq 0$, we write $\Gamma(2^n) \lhd \Sp(T_2(J))$ for the level-$2^n$ principal congruence subgroup consisting of automorphisms whose images modulo $2^n$ are trivial.

It is well known that we always have $K_1 \subseteq k(\alpha_1, ... , \alpha_d)$ and that equality holds except when $d = 4$ (see \cite{yelton2017note}).  The main purpose of this paper is to provide an explicit description of the maximal abelian subextension of $K_{\infty} / K_1$, which we denote by $K_{\infty}^{\ab}$.  (Below for any integer $n \geq 1$, we write $\zeta_{2^n} \in \bar{k}$ to denote a $2^n$th root of unity.)

\begin{thm} \label{thm main}

For $1 \leq i, j \leq 2g + 1$, let $\gamma_{i, j} = \alpha_j - \alpha_i$ (resp. $\gamma_{i, j} = (\alpha_j - \alpha_i) \prod_{l \neq i, j} (\alpha_d - \alpha_l)$) if $d = 2g + 1$ (resp. if $d = 2g + 2$).

a) Suppose that $g \geq 2$.  If $\zeta_4 \in k$, we have 
\begin{equation} \label{eq main (a)} K_2(\mu_2) = K_1(\mu_2, \{\sqrt{\gamma_{i, j}}\}_{1 \leq i < j \leq 2g + 1}) \subsetneq K_{\infty}^{\ab} = K_2(\mu_2, \{\sqrt[4]{\scriptstyle\prod_{j \neq i} \gamma_{i, j}}\}_{i = 1}^{2g + 1}) \subsetneq K_3(\mu_2) \end{equation}
 and $\Gal(K_{\infty}^{\ab} / K_1(\mu_2)) \cong (\zz / 2\zz)^{2g^2 - g} \times (\zz / 4\zz)^{2g}$.  If $\zeta_4 \notin k$, we instead have $K_{\infty}^{\ab} = K_2(\mu_2)$.

b) Suppose that $g = 1$.  If $\zeta_8 \in k$, we have 
\begin{equation} \label{eq main (b)} K_2(\mu_2) = K_1(\mu_2, \{\sqrt{\gamma_{i, j}}\}_{1 \leq i < j \leq 3}) \subsetneq K_{\infty}^{\ab} = K_2(\mu_2, \sqrt[8]{\gamma_{1, 2} \gamma_{1, 3} \gamma_{2, 3}^2}, \sqrt[8]{\gamma_{2, 3} \gamma_{2, 1} \gamma_{3, 1}^2}, \sqrt[8]{\gamma_{3, 1} \gamma_{3, 2} \gamma_{1, 2}^2}) \subsetneq K_4(\mu_2) \end{equation}
 and $\Gal(K_{\infty}^{\ab} / K_1(\mu_2)) \cong \zz / 2\zz \times (\zz / 8\zz)^2$.

If $\zeta_8 \notin k$ but $\zeta_4 \in k$, we instead have 
\begin{equation} K_2(\mu_2) = K_1(\mu_2, \{\sqrt{\gamma_{i, j}}\}_{1 \leq i < j \leq 3}) \subsetneq K_{\infty}^{\ab} = K_2(\mu_2, \sqrt[4]{\gamma_{1, 2} \gamma_{1, 3}}, \sqrt[4]{\gamma_{2, 3}\gamma_{2, 1}}, \sqrt[4]{\gamma_{3, 1}\gamma_{3, 2}}) \subsetneq K_3(\mu_2) \end{equation}
 and $\Gal(K_{\infty}^{\ab} / K_1(\mu_2)) \cong \zz / 2\zz \times (\zz / 4\zz)^2$.

Finally, if $\zeta_4 \notin k$, we instead have $K_{\infty}^{\ab} = K_2(\mu_2) = K_1(\mu_2, \{\sqrt{\gamma_{i, j}}\}_{1 \leq i < j \leq 3})$ and $\Gal(K_{\infty}^{\ab} / K_1(\mu_2)) \cong (\zz / 2\zz)^3$.

c) Let $\sigma \in G_K$ be any Galois automorphism such that $\rho_2(\sigma) = -1 \in \GSp(T_{2}(J))$.  Then $\sigma$ acts on $K_{\infty}^{\ab}$ by fixing $K_1(\mu_2)$, changing the signs of all generators of the form $\sqrt{\gamma_{i, j}}$, and fixing (resp. changing the signs of) the remaining generators given in (\ref{eq main (a)}) and (\ref{eq main (b)}) if $g$ is even (resp. if $g$ is odd).

\end{thm}

The following corollary is proven via the argument in Step 4 of the proof of \cite[Lemma 3]{yelton2017note}.

\begin{cor} \label{cor specialization}

Let $a_0, ... , a_{d - 1} \in k$ be elements such that $f(x) := x^d + \sum_{i = 0}^{d - 1} a_i x^i \in k[x]$ is separable, and let $\bar{J}$ be the Jacobian of the hyperelliptic curve defined over $k$ by the equation $y^2 = f(x)$.  Let $\bar{\alpha}_1, ... , \bar{\alpha}_d \in \bar{k}$ denote the roots of the polynomial $f(x) \in k[x]$ and let $\bar{\gamma}_{i, j} \in \bar{k}$ be given by formulas in terms of the $\bar{\alpha}_i$'s analogous to those used to define the $\gamma_{i, j}$'s in the statement of Theorem \ref{thm main}.

a) If $g \geq 2$, the extension $k(\bar{J}[8]) / k(\bar{J}[2])$ contains the subextension 
\begin{equation} \label{eq specialization (a)} k(\bar{J}[2])(\{\sqrt{\bar{\gamma}_{i, j}}\}_{1 \leq i < j \leq 2g + 1}, \{\sqrt[4]{\scriptstyle\prod_{j \neq i} \bar{\gamma}_{i, j}}\}_{i = 1}^{2g + 1}, \zeta_8). \end{equation}

b) If $g = 1$, the extension $k(\bar{J}[16]) / k(\bar{J}[2])$ contains the subextension 
\begin{equation} \label{eq specialization (b)} k(\bar{J}[2])(\{\sqrt{\bar{\gamma}_{i, j}}\}_{1 \leq i < j \leq 2g + 1}, \sqrt[8]{\bar{\gamma}_{1, 2} \bar{\gamma}_{1, 3} \bar{\gamma}_{2, 3}^2}, \sqrt[8]{\bar{\gamma}_{2, 3} \bar{\gamma}_{2, 1} \bar{\gamma}_{3, 1}^2}, \sqrt[8]{\bar{\gamma}_{3, 1} \bar{\gamma}_{3, 2} \bar{\gamma}_{1, 2}^2}, \zeta_{16}). \end{equation}

c) Let $\sigma$ be any automorphism in the absolute Galois group of $k$ which acts on $k(\bar{J}[16])$ as multiplication by $-1$.  Then $\sigma$ acts on the subfields described above by changing the signs of all generators of the form $\sqrt{\gamma_{i, j}}$ and by fixing (resp. changing the signs of) all remaining generators given in (\ref{eq specialization (a)}) and (\ref{eq specialization (b)}) if $g$ is even (resp. if $g$ is odd).

\end{cor}

\begin{rmk} \label{rmk dyadic elliptic}

We can also verify part (b) of Theorem \ref{thm main} for the $d = 3$ case by combined use of the formulas given in \cite{yelton2015dyadic} and \cite{yelton2017note}.  We illustrate how to see that $K_4(\mu_2)$ contains an element whose $8$th power is $\gamma_{1, 2}\gamma_{1, 3}\gamma_{2, 3}^2$ as follows (one may use a similar argument for the other generators).  For $n = 1, 2, 3$ and $1 \leq i < j \leq 3$, we fix elements $\sqrt[2^n]{\gamma_{i, j}} \in \bar{K}$ whose $2^n$th powers are $\gamma_{i, j} \in \bar{K}$ and which are compatible in the obvious way, and we fix a square root of $\sqrt{\gamma_{1, 2}} + \sqrt{\gamma_{1, 3}}$.  (Note that due to the equivariance of the Weil pairing, $-1$ has a $2^n$th root in $K(J[2^{n + 1}])$ for each $n \geq 1$.)  Let $\mathcal{L}$ be the $3$-regular tree defined in \cite{yelton2015dyadic}, and assume the notation used throughout that paper.  Let $\{\Lambda_0, \Lambda_1, \Lambda_2, \Lambda_3, \Lambda_4\}$ be a non-backtracking path in $\mathcal{L}$, where $\Lambda_0$ is the root and $\Lambda_1 = \Lambda(\alpha_1)$.  Then it is tedious but straightforward to verify that there exists a decoration $\Psi : \mathcal{L} \setminus \{\Lambda_0\} \to \bar{K}$ (see \cite[Definition 1.2]{yelton2015dyadic}) such that $\Psi(\Lambda_2), \Psi(\Lambda_2') = \gamma_{1, 2} + \gamma_{1, 3} \pm 2\sqrt{\gamma_{1, 2}}\sqrt{\gamma_{1, 3}}$; 
$$\Psi(\Lambda_3), \Psi(\Lambda_3') = -\Psi(\Lambda_2) - (\Psi(\Lambda_2) - \Psi(\Lambda_2')) \pm 4(\sqrt{\gamma_{1, 2}} + \sqrt{\gamma_{1, 3}})\sqrt[4]{\gamma_{1, 2}}\sqrt[4]{\gamma_{1, 3}};$$
 and 
$$\Psi(\Lambda_4) = -\Psi(\Lambda_3) - (\Psi(\Lambda_3) - \Psi(\Lambda_3')) + 4\sqrt{2}(\sqrt{\gamma_{1, 2}} + \sqrt{\gamma_{1, 3}} + 2\sqrt[4]{\gamma_{1, 2}}\sqrt[4]{\gamma_{1, 3}})\sqrt{\sqrt{\gamma_{1, 2}} + \sqrt{\gamma_{1, 3}}}\sqrt[8]{\gamma_{1, 2}}\sqrt[8]{\gamma_{1, 3}}.$$
By \cite[Proposition 2.5(b)]{yelton2015dyadic}, we have $\Psi(\Lambda_2), \Psi(\Lambda_2'), \Psi(\Lambda_3), \Psi(\Lambda_3'), \Psi(\Lambda_4) \in K_4$.  Moreover, from \cite[Theorem 1, Remark 11(b)]{yelton2017note}, we see that $\sqrt{2}(\sqrt{\gamma_{1, 2}} + \sqrt{\gamma_{1, 3}} + 2\sqrt[4]{\gamma_{1, 2}}\sqrt[4]{\gamma_{1, 3}}) \in K_3$.  It follows that 
\begin{equation} \sqrt{\sqrt{\gamma_{1, 2}} + \sqrt{\gamma_{1, 3}}}\sqrt[8]{\gamma_{1, 2}}\sqrt[8]{\gamma_{1, 3}} \in K_4. \end{equation}
  By \cite[Theorem 1]{yelton2017note}, we have $\sqrt{\pm\gamma_{i, j}}, B_1 := \sqrt[4]{-\gamma_{2, 3}} \sqrt{\sqrt{\gamma_{1, 2}} + \sqrt{\gamma_{1, 3}}} \in K_3$.  Therefore, we have 
\begin{equation} \sqrt[4]{-\gamma_{2, 3}}\sqrt[8]{\gamma_{1, 2}}\sqrt[8]{\gamma_{1, 3}} = \sqrt{-\gamma_{2, 3}}\sqrt{\sqrt{\gamma_{1, 2}} + \sqrt{\gamma_{1, 3}}}\sqrt[8]{\gamma_{1, 2}}\sqrt[8]{\gamma_{1, 3}} / B_1 \in K_4. \end{equation}

\end{rmk}

The rest of this paper is dedicated to a proof of Theorem \ref{thm main}; our plan is as follows.  We will first assume that $k = \cc$ and prove Theorem \ref{thm main} in that case by viewing the situation in a topological setting similar to the author's strategy in \cite{yelton2015images}; we will retain this assumption throughout \S\ref{sec2} and \S\ref{sec3}.  In \S\ref{sec2}, we determine generators for the $4$-torsion field $K_2$, which is contained in $K_{\infty}^{\ab}$.  Then in \S\ref{sec3}, we determine generators for $K_{\infty}^{\ab}$ over $K_2$, treating the $g \geq 2$ case and the $g = 1$ case separately.  Finally, in \S\ref{sec4} we generalize these results to the situation where $k$ is any field of characteristic different from $2$.

The author would like to thank the referee for a number of corrections and suggestions which have improved this text.

\section{The $4$-division field over $\cc$} \label{sec2}

We assume for this section as well as in \S\ref{sec3} that $k = \cc$, so that $K$ is generated over $\cc$ by the symmetric functions of the transcendental elements $\alpha_i$.  We will consider $K$ as a subfield of the function field of the ordered configuration space $Y_d$ of $d$-element ordered subsets of $\cc$; we view $Y_d(\cc)$ as a topological space.  The fundamental group of $Y_d$ is well known to be the pure braid group on $d$ strands, which we denote by $P_d$.  This has a well-known presentation (see \cite[Lemma 1.8.2]{birman1974braids}) with generators $A_{i, j}$ for $1 \leq i < j \leq d$.  It is known (see \cite[Corollary 1.8.4]{birman1974braids} and its proof) that the center of $P_d$ is cyclically generated by the element $\Sigma := A_{1, 2} (A_{1, 3} A_{2, 3}) ... (A_{1, d} A_{2, d} ... A_{d - 1, d}) \in P_d$.  The profinite completion $\widehat{P}_d$ of $P_d$ is the \'{e}tale fundamental group of $Y_d$ and may be identified with the Galois group of $K^{\unr} / K(\{\alpha_i\}_{1 \leq i \leq d})$, where $K^{\unr}$ is the maximal extension of $K$ unramified away from the primes $(\alpha_j - \alpha_i)$ for $1 \leq i < j \leq d$.  The criterion of N\'{e}ron-Ogg-Shafarevich (\cite[Theorem 1]{serre1968good}) implies that the natural $\ell$-adic representation $\rho_{\ell} : G_K \to \GSp(T_{\ell}(J))$, restricted to the subgroup fixing the Galois extension $K(\{\alpha_i\}_{1 \leq i \leq d})$, factors through the restriction map $\Gal(\bar{K} / K(\{\alpha_i\}_{1 \leq i \leq d})) \twoheadrightarrow \widehat{P}_d$; we denote the induced representation of $\widehat{P}_d$ also by $\rho_{\ell}$.

There is a ``universal" family of hyperelliptic curves $\mathcal{C} \to Y_d$ whose fiber $\mathcal{C}_{\underline{z}}$ over each point $\underline{z} = (z_1, ... , z_d) \in Y_d(\cc)$ is the hyperelliptic curve given by the monic polynomial in $\cc[x]$ whose roots are the elements of the $d$-element ordered set $\underline{z}$; this family has $C$ as its generic fiber.  We write $\rho^{\tp} : P_d \to \Aut(H_1(C_{\underline{z}_0}, \zz))$ for the representation induced by the monodromy representation $P_d \cong \pi_1(Y_d, \underline{z}_0) \to \Aut(\pi_1(\mathcal{C}_{\underline{z}_0}, P_0))$ associated to the family $\mathcal{C} \to Y_d$, where $\underline{z}_0 := (1, ... , d) \in Y_d(\cc)$ and $P_0 \in \mathcal{C}_{\underline{z}}$ are basepoints).  The monodromy action respects the intersection pairing on $\mathcal{C}_{\underline{z}_0}$, and therefore, the image of $\rho^{\tp}$ is contained in the group of symplectic automorphisms $\Sp(H_1(C_{\underline{z}_0}, \zz))$.  (See \cite[\S2]{yelton2015images} for more details of this construction.)

As both $P_d$ and $\Sp(H_1(C_{\underline{z}_0}, \zz))$ are residually finite, the representation $\rho^{\tp}$ induces a representation of the profinite completion $\widehat{P}_d$ on each pro-$\ell$ completion $H_1(\mathcal{C}_{\underline{z}_0}, \zz) \otimes \zz_{\ell}$ of $H_1(\mathcal{C}_{\underline{z}_0}, \zz)$.  For each prime $\ell$, we denote this representation by $\rho_{\ell}^{\tp} : \widehat{P}_d \to \Sp(H_1(\mathcal{C}_{\underline{z}_0}, \zz) \otimes \zz_{\ell})$.  Our technique is to study $\rho_2$ by relating it to the topologically-defined representation $\rho_2^{\tp}$ using a key comparison result proved by the author as \cite[Proposition 2.2]{yelton2015images}.

\begin{lemma} \label{lemma key}

For any prime $\ell$, there is an isomorphism of $\zz_{\ell}$-modules $H_1(\mathcal{C}_{\underline{z}_0}, \zz) \otimes \zz_{\ell} \stackrel{\sim}{\to} T_{\ell}(J)$ making the representations $\rho_{\ell}^{\tp}$ and $\rho_{\ell}$ isomorphic.

\end{lemma}

We now state and prove some properties of the representation $\rho^{\tp}$ that we will need below.

\begin{prop} \label{prop rho top}

a) The image of $P_d$ under $\rho^{\tp}$ coincides with the principal congruence subgroup $\Gamma(2) \lhd \Sp(H_1(\mathcal{C}_{\underline{z}_0}, \zz))$.

b) If $d$ is odd, we have $\rho^{\tp}(\Sigma) = -1 \in \Gamma(2)$.

c) If $d$ is even, we have $\rho^{\tp}(\Sigma) = 1 \in \Gamma(2)$ and $\rho^{\tp}(\Sigma') = -1 \in \Gamma(2)$, where 
$$\Sigma' = A_{1, 2} (A_{1, 3} A_{2, 3}) ... (A_{1, d - 1} A_{2, d - 1} ... A_{d - 2, d - 1}) \in P_d.$$

\end{prop}

\begin{proof}

The statement of (a) has been shown in several works: see \cite[Th\'{e}or\`{e}me 1]{a1979tresses}, \cite[Lemma 8.12]{mumford1984tata}, or \cite[Theorem 7.3(ii)]{yu351toward}.

Assume that $d$ is odd, so $d = 2g + 1$, where $g$ is the genus of $\mathcal{C}_{\underline{z}_0}$.  Then one deduces directly from the presentation of $P_d$ given by \cite[Lemma 1.8.2]{birman1974braids} that the abelianization of $P_d$ is a free $\zz$-module of rank $2g^2 + g$ whose generators are the images of the elements $A_{i, j}$, $1 \leq i < j \leq 2g + 1$; its maximal abelian quotient of exponent $2$ is therefore a $(2g^2 + g)$-dimensional $\ff_2$-vector space generated by the images of the $A_{i, j}$'s.  Meanwhile, it follows directly from \cite[Corollary 2.2]{sato2010abelianization} that the maximal abelian exponent-$2$ quotient of $\Gamma(2)$ also has rank $2g^2 + g$.  It follows that $\rho^{\tp}$ induces an isomorphism between the exponent-$2$ abelianizations of $P_d$ and $\Gamma(2)$.  Now since $\Sigma$ is a product of each of the elements $A_{i, j} \in P_d$, it has nontrivial image in the exponent-$2$ abelianization of $P_d$ and therefore has nontrivial image in $\Gamma(2)$.  Meanwhile, as $\Sigma$ lies in the center of $P_d$ and $\rho^{\tp}(P_d) = \Gamma(2)$ by (a), the image $\rho^{\tp}(\Sigma)$ lies in the center of $\Gamma(2)$.  The only nontrivial central element of $\Gamma(2)$ is the scalar $-1$, proving part (b).

Now assume that $d$ is even.  Note that the family $\mathcal{C} \to Y_d$ is an unramified degree-$2$ cover of the family $Y_{d + 1}' \to Y_d$ whose fiber over each $\underline{z} = (z_1, ... , z_d) \in Y_d$ is $\proj_{\cc}^1 \smallsetminus \{z_1, ... , z_d\}$ ($Y_{d + 1}'$ is essentially the ordered configuration space of cardinality-$(d + 1)$ subsets of $\proj_{\cc}^1$ whose first $d$ elements lie in $\cc$).  This implies that the monodromy action $\rho^{\tp}$ is induced by the monodromy action associated to the family $Y_{d + 1}' \to Y_d$ via the inclusion of and quotients by characteristic subgroups 
\begin{equation} \pi_1(\proj_{\cc}^1(\cc) \smallsetminus \{z_i\}_{1 \leq i \leq d}, \bar{P}_0) \rhd \pi_1(\mathcal{C}_{\underline{z}_0}(\cc) \smallsetminus \{(z_i, 0)\}_{1 \leq i \leq d}, P_0) \twoheadrightarrow \pi_1(\mathcal{C}_{\underline{z}_0}(\cc), P_0) \end{equation}
(here $\bar{P}_0$ is the projection of the baspoint $P_0 \in \mathcal{C}_{\underline{z}_0}$).  In fact, if we let $x_1, ... , x_d$ denote the generators of $\pi_1(\proj_{\cc}^1(\cc) \smallsetminus \{z_i\}_{1 \leq i \leq d}, \bar{P}_0)$ given in \cite[\S4]{hasson2017prime}, then $\pi_1(\mathcal{C}_{\underline{z}_0}(\cc) \smallsetminus \{(z_i, 0)\}_{1 \leq i \leq d}, P_0)$ is the subgroup generated by the elements $x_i x_{i + 1}$ for $1 \leq i \leq d - 1$ and $x_j^2$ for $1 \leq j \leq d$, and $\pi_1(\mathcal{C}_{\underline{z}_0}(\cc), P_0)$ is the quotient of this by the elements $x_j^2$.  Using the fact that the images of the elements $x_i x_{i + 1}$ for $1 \leq i \leq d - 2$ form a $\zz$-basis of the abelianization $H_1(\mathcal{C}_{\underline{z}_0}, \zz)$, one may then explicitly compute the automorphisms of $H_1(\mathcal{C}_{\underline{z}_0}, \zz)$ induced by $\Sigma$ and $\Sigma'$ using the statement and proof of \cite[Lemma 4.1]{hasson2017prime}, thus verifying part (c).  (One can also prove that $\rho^{\tp}(\Sigma) = 1$ using \cite[Lemma 4.2]{hasson2017prime} and the well-known fact that the kernel of Birman's surjection onto the mapping class group coincides with the center of $P_d$.)

\end{proof}

Below, for each integer $N \geq 1$, we fix $\zeta_N \in \cc$ to be the $N$th root of unity given by $e^{2 \pi \sqrt{-1} / N}$.

\begin{lemma} \label{lemma pure braid action abelian}

For any integer $N \geq 1$, the maximal abelian exponent-$N$ subextension of $K^{\unr} / K(\{\alpha_i\}_{1 \leq i \leq d})$ coincides with $K_1(\{\sqrt[N]{\alpha_j - \alpha_i}\}_{1 \leq i < j \leq d})$.  Each standard generator $A_{i, j}$ of $P_d \subset \widehat{P}_d = \Gal(K^{\unr} / K(\{\alpha_i\}_{1 \leq i \leq d})$ acts on this subextension by sending $\sqrt[N]{\alpha_j - \alpha_i}$ to $\zeta_N\sqrt[N]{\alpha_j - \alpha_i}$ and fixing all of the other generators.
\end{lemma}

\begin{proof}

The first statement results from a standard application of Kummer theory.

To prove the second statement, we fix some $N \geq 1$ and assume without loss of generality that $(i, j) = (d - 1, d)$.  Let $Y_d^{(N)} \to Y_d$ be the covering corresponding to the maximal abelian exponent-$N$ subextension of $K^{\unr} / K_1$.  We choose as a topological representative of $A_{i, j}$ the loop given by $t \mapsto (1, ... , d - 1, d - 1 + e^{2\pi\sqrt{-1}t} \varepsilon/2) \in Y_d(\cc)$ for $t \in [0, 1]$, where $\varepsilon$ is a real number satisfying $\min_{1 \leq i \leq d - 2} |z_{d - 1} - z_i| > \varepsilon > 0$ (we are allowed to move the basepoint to $(1, ... , d - 1, d - 1 + \varepsilon/2) \in Y_d(\cc)$ because the monodromy action we are considering factors through the abelianization of the fundamental group).  We have a closed embedding of the punctured disk $B^* := \{z \in \cc \ | \ 0 < |z| < 1\}$ into $Y_d(\cc)$ given by $z \mapsto (1, ... , d - 1, d - 1 + \varepsilon z)$ which takes a loop representing the standard generator of $\pi_1(B^*, 1/2)$ to the loop representing $A_{d - 1, d}$ defined above.  The pullback of the cover $Y_d^{(N)}(\cc) \to Y_d(\cc)$ via $B^* \hookrightarrow Y_d(\cc)$ is clearly homeomorphic to the cover $B^* \to B^*$ given by $z \mapsto z^N$.  Locally, the standard generator of $\pi_1(B^*, 1/2)$ acts on the ring of holomorphic functions defined on the covering space as $\sqrt[N]{z} \mapsto \zeta_N \sqrt[N]{z}$, and the claim follows.

\end{proof}

We are now ready to state and prove the main result of this section which explicitly describes the extension $K_2 / K$ and shows in particular that it is abelian over $K_1$ and therefore contained in $K_{\infty}^{\ab}$.

\begin{thm} \label{thm 4-torsion}

For $1 \leq i, j \leq 2g + 1$, let $\gamma_{i, j}$ be defined as in the statement of Theorem \ref{thm main}.  Then we have $K_2 = K_1(\{\sqrt{\gamma_{i, j}}\}_{1 \leq i < j \leq 2g + 1})$.  Moreover, any Galois automorphism in $G_K$ which acts as multiplication by $-1$ on the subgroup $J[4]$ acts on $K_2$ by fixing $K_1$ and changing the signs of all generators $\sqrt{\gamma_{i, j}}$.

\end{thm}

\begin{proof}

For $g = 1$, this was already shown by the author as \cite[Proposition 6(a),(b)]{yelton2017note}, so we assume that $g \geq 2$.  In particular, this means that $K_1 = \cc(\{\alpha_i\}_{1 \leq i \leq d})$.  The first statement of the theorem was proved for odd $d$ as \cite[Proposition 3.1]{yelton2015images}, but the following argument proves the full theorem for general $d$.

By \cite[Corollary 1.2(c)]{yelton2015images}, we have that $\rho_2$ induces an isomorphism $\bar{\rho}_4 : \Gal(K_2 / K_1) \stackrel{\sim}{\to} \Gamma(2) / \Gamma(4)$.  We note from the proof of \cite[Corollary 2.2]{sato2010abelianization} that the largest abelian quotient of $\Gamma(2)$ of exponent $2$ is in fact $\Gamma(2) / \Gamma(4) \cong (\zz / 2\zz)^{2g^2 + g}$; therefore, $K_2$ is the maximal abelian subextension of $K_{\infty} / K_1$ of exponent $2$.  Since the extension $K_{\infty} / K_1$ is unramified over all primes of the form $(\alpha_j - \alpha_i)$, such a subextension must be a subfield of $\widetilde{K}_2 := K_1(\{\sqrt{\alpha_j - \alpha_i}\}_{1 \leq i < j \leq d})$.  If $d = 2g + 1$, then $\Gal(\widetilde{K}_2 / K_1)$ already has rank $2g^2 + g$ and therefore $K_2 = \widetilde{K}_2$; moreover, Proposition \ref{prop rho top}(b) combined with Lemma \ref{lemma pure braid action abelian} implies that any Galois element whose image under $\bar{\rho}_4$ is $-1 \in \Gamma(2) / \Gamma(4)$ changes the sign of each $\sqrt{\alpha_j - \alpha_i} = \sqrt{\gamma_{i, j}}$.  This proves the statement in the $d = 2g + 1$ case.

Now suppose that $d = 2g + 2$.  Then $\Gal(\widetilde{K}_2 / K_1)$ has rank $d(d - 1)/2 > 2g^2 + g$, which implies that $K_2 \subsetneq \widetilde{K}_2$.  Kummer theory then tells us that $\Gal(K_2 / K_1)$ is canonically identified with the dual of some subgroup $H \subset K_1^{\times} / (K_1^{\times})^2$, where both are considered as vector spaces over $\ff_2$ of dimension $2g^2 + g$.  Clearly $H$ is a subspace of the space $\widetilde{H} \subset  K_1^{\times} / (K_1^{\times})^2$ (which itself is the dual of $\Gal(\widetilde{K}_2 / K_1)$) generated by images of the elements $\alpha_j - \alpha_i \in K_1^{\times}$; for $1 \leq i < j \leq d$; we write $[\alpha_j - \alpha_i] = [\alpha_i - \alpha_j] \in \widetilde{H}$ for each such image and use additive notation for elements of $\widetilde{H}$.

We now identify $\widetilde{H}$ with its dual $\Gal(\widetilde{K}_2 / K_1)$ via the basis $\{[\alpha_j - \alpha_i]\}_{1 \leq i < j \leq d}$.  Lemma \ref{lemma pure braid action abelian} implies that the image of each $A_{i, j} \in P_d \subset \widehat{P}_d = \Gal(K^{\unr} / K_1)$ under $\rho_2$ composed with reduction modulo $4$ is the one induced by $[\alpha_j - \alpha_i] \in \Gal(\tilde{K}_2 / K_1) = \tilde{H}$.  It follows from Proposition \ref{prop rho top}(c) that each element of $H$ must be the sum of an \textit{even} number of generators of $\widetilde{H}$.

We note that the degree-$(d - 1)$ curve $C' : y'^2 = \prod_{i = 1}^{2g + 1} (x' - 1 / (\alpha_{2g + 2} - \alpha_i))$ is isomorphic to the degree-$d$ curve $C : y^2 = \prod_{i = 1}^{2g + 2} (x - \alpha_i)$ over the quadratic extension $K(\beta) / K$ via the change of variables 
$$(x', y') = (1 / (\alpha_{2g + 2} - x), y / (\beta(\alpha_{2g + 2} - x)^{g + 1})),$$
 where $\beta \in \bar{K}$ is a square root of the element $\prod_{i = 1}^{2g + 1} (\alpha_i - \alpha_{2g + 2})$.  (This is just the isomorphism of hyperelliptic curves induced from an automorphism of the projective line which moves $\alpha_{2g + 2}$ to $\infty$.)  From what was shown above for the odd-degree case, we have 
\begin{equation} K_2(\beta) = K_1(\beta, \{\sqrt{(\alpha_{2g + 2} - \alpha_j)^{-1} - (\alpha_{2g + 2} - \alpha_i)^{-1}}\}_{1 \leq i < j \leq 2g + 1}) = K_1(\beta, \{\beta \sqrt{\gamma_{i, j}}\}_{1 \leq i < j \leq 2g + 1}). \end{equation}
Thus, $K_2(\beta)$ is generated over $K_1(\{\sqrt{\gamma_{i, j}}\}_{1 \leq i < j \leq 2g + 1})$ by the element $\beta$.  Since each $\gamma_{i, j}$ corresponds to the element $[\alpha_j - \alpha_i] + \sum_{l \neq i, j} [\alpha_{2g + 2} - \alpha_l] \in H$, which is the sum of an even number of generators while $\beta^2$ corresponds to $\sum_{i \neq 2g + 2} [\alpha_{2g + 2} - \alpha_i] \in H$, which is not, the extension $K_1(\{\sqrt{\gamma_{i, j}}\}_{1 \leq i < j \leq 2g + 1}) / K_1$ must be the fixed field corresponding to $H \subset \widetilde{H}$ and therefore coincides with $K_2$.  Moreover, if $\Sigma' \in P_{2g + 2}$ is the braid defined as in the statement of Proposition \ref{prop rho top}(c), then that proposition says that $\Sigma'$ corresponds to a Galois automorphism whose restriction to $K_{\infty}$ is $-1 \in \Gamma(2) \cong \Gal(K_{\infty} / K_1)$.  Then Lemma \ref{lemma pure braid action abelian} implies that $\Sigma'$ changes the sign of each $\sqrt{\gamma_{i, j}}$, thus implying that any Galois element whose image under $\bar{\rho}_4$ is $-1 \in \Gamma(2) / \Gamma(4)$ acts in this way, hence the statement in the $d = 2g + 2$ case.

\end{proof}

\section{The maximal abelian subfield over $\cc$} \label{sec3}

\subsection{The abelianization of the Galois group} \label{sec3.1}

We retain our assumption from the last section that $k = \cc$.  Having found a particular abelian subextension of $K_{\infty} / K_1$, namely $K_2 / K_1$, we shall now determine the maximal abelian subextension.  In order to do this, we first need to know what its Galois group over $K_1$ looks like.

\begin{lemma} \label{lemma abelianization}

The abelianization $\Gamma(2)^{\ab}$ of the principal congruence subgroup $\Gamma(2) \lhd \Sp_{2g}(\zz_2)$ is isomorphic to $(\zz / 2\zz)^{2g^2 - g} \times (\zz / 4\zz)^{2g}$ (resp. $\zz / 2\zz \times (\zz / 8\zz)^2$), and the abelianization map $\pi: \Gamma(2) \twoheadrightarrow \Gamma(2)^{\ab}$ factors through $\Gamma(2) / \Gamma(8)$ (resp. $\Gamma(2) / \Gamma(16)$) if $g \geq 2$ (resp. if $g = 1$).
\end{lemma}

\begin{proof}

The description of $\Gamma(2)^{\ab}$ for the $g \geq 2$ case is given by \cite[Corollary 2.2]{sato2010abelianization}.  We therefore assume that $g = 1$ and proceed to compute the commutator subgroup $[\Gamma(2), \Gamma(2)] \lhd \Gamma(2)$.

We first claim that $[\Gamma(2), \Gamma(2)]$ contains $\Gamma(16)$.  Write $\sigma = \begin{bmatrix} 1 & -2 \\ 0 & 1 \end{bmatrix}$ and $\tau = \begin{bmatrix} 1 & 0 \\ 2 & 1 \end{bmatrix}$.  We verify by straightforward computation that for any integers $m, n \geq 1$, we have the formula 
\begin{equation} \label{eq commutator} \sigma^m \tau^n \sigma^{-m} \tau^{-n} = \begin{bmatrix} 1 - 2^{m + n} & 2^{2m + n} \\ -2^{m + 2n} & 1 + 2^{m + n} + 2^{2m + 2n} \end{bmatrix} \in [\Gamma(2), \Gamma(2)]. \end{equation}
Using this formula, we compute $(\sigma^2\tau\sigma^{-2}\tau^{-1})(\sigma\tau\sigma^{-1}\tau^{-1})^2 \equiv \tau^8$, $(\sigma\tau^2\sigma^{-1}\tau^{-2})(\sigma\tau\sigma^{-1}\tau^{-1})^2 \equiv \sigma^8$, and $(\sigma\tau\sigma^{-1}\tau^{-1})^4 \equiv 17$ modulo $32$.  It is easy to show that for $n \geq 1$, the images modulo $2^{n + 1}$ of $\sigma^{2^{n - 1}}$, $\tau^{2^{n - 1}}$, and the scalar matrix $1 + 2^n$ generate $\Gamma(2^n) / \Gamma(2^{n + 1}) \cong (\zz / 2\zz)^3$; thus, in particular, $[\Gamma(2), \Gamma(2)]$ contains $\Gamma(16)$ modulo $32$.  Now we show by induction that for each $n \geq 5$, $[\Gamma(2), \Gamma(2)]$ contains $\Gamma(2^n)$ modulo $2^{n + 1}$, which suffices to prove that $[\Gamma(2), \Gamma(2)] \supset \Gamma(16)$.  Assume this is the case for $n - 1$; then in particular, $[\Gamma(2), \Gamma(2)]$ contains elements which are equivalent modulo $2^n$ to $\sigma^{2^{n - 2}}$, $\tau^{2^{n - 2}}$, and $1 + 2^{n - 1}$.  On computing that the squares of such elements must be equivalent modulo $2^{n + 1}$ to $\sigma^{2^{n - 1}}$, $\tau^{2^{n - 1}}$, and $1 + 2^n$ respectively, the claim is proven.

We next claim that the image of $[\Gamma(2), \Gamma(2)]$ modulo $16$ is cyclically generated by the image modulo $16$ of $\sigma\tau\sigma^{-1}\tau^{-1}$.  To see this, we recall the well-known fact that $\Gamma(2)$ decomposes as the direct product of $\{\pm 1\}$ and the subgroup generated by $\sigma$ and $\tau$ and therefore, $[\Gamma(2), \Gamma(2)]$ coincides with the commutator subgroup of $\langle \sigma, \tau \rangle \lhd \Gamma(2)$.  On checking that $\sigma\tau\sigma^{-1}\tau^{-1}$ commutes with both $\sigma$ and $\tau$ modulo $16$, we deduce as an easy exercise in group theory that the commutator of any two elements in $\Gamma(2) / \Gamma(16)$ is a power of the image of $\sigma\tau\sigma^{-1}\tau^{-1}$.  Since the smallest normal subgroup of $\Gamma(2) / \Gamma(16)$ containing these powers is simply the cyclic subgroup generated by the image of $\sigma\tau\sigma^{-1}\tau^{-1}$, we have proven the claim.

Now it follows from the fact that $\Gamma(2^n) / \Gamma(2^{n + 1}) \cong (\zz / 2\zz)^3$ for $n \geq 1$ that $\Gamma(2) / \Gamma(16)$ has order $2^9 = 512$; meanwhile, we see from what we have computed above that the image of $\sigma\tau\sigma^{-1}\tau^{-1}$ modulo $16$, which generates the image of $[\Gamma(2), \Gamma(2)]$, has order $4$.  Therefore, $\Gamma(2)^{\ab}$ has order $128$.  Since $\Gamma(2)^{\ab}$ is generated by the images of $-1$, $\sigma$, and $\tau$ modulo $[\Gamma(2), \Gamma(2)]$, the first statement of the lemma follows from an easy verification (by considering the image of $[\Gamma(2), \Gamma(2)]$ modulo $16$) that the images of $\sigma$ and $\tau$ each have order $8$ in $\Gamma(2)^{\ab}$.

\end{proof}

In order to find the extension of $K_1$ corresponding to the Galois quotient described by the lemma, we consider the $g \geq 2$ and $g = 1$ cases separately.

\subsection{The $g \geq 2$ case}

Lemma \ref{lemma abelianization}, together with the results of \S\ref{sec2} and the fact that $K_{\infty} \subset K^{\unr}$, implies that $K_{\infty}^{\ab}$ is an extension of $K_2 = K_1(\{\sqrt{\gamma_{i, j}}\}_{1 \leq i < j \leq 2g + 1})$ obtained by adjoining $2g$ independent $4$th roots of products of the elements $(\alpha_j - \alpha_i)$.  Similarly to what we saw in the proof of Theorem \ref{thm 4-torsion}, Kummer theory tells us that $\Gal(K_{\infty}^{\ab} / K_2)$ is canonically identified with the dual of some subgroup $V \subset K_2^{\times} / (K_2^{\times})^2$, where both are considered as vector spaces over $\ff_2$ of dimension $2g$.  In fact, since $K_{\infty} \subset K^{\unr}$, the first statement of Lemma \ref{lemma pure braid action abelian} implies that $V$ is also a subspace of the space $\widetilde{V}$ generated by the images in $\widetilde{K}_2^{\times} / (\widetilde{K}_2^{\times})^2$ of elements of the form $\sqrt{\alpha_j - \alpha_i}$, where $\widetilde{K}_2 = K_2(\{\sqrt{\alpha_j - \alpha_i}\}_{1 \leq i < j \leq d})$ as in the proof of Theorem \ref{thm 4-torsion}.  For ease of notation, we denote each of these images by $[i, j] = [j, i] \in \widetilde{V}$.  We now proceed to explicitly determine the $2g$-dimensional subspace $V \subset K_2^{\times} / (K_2^{\times})^2 \cap \widetilde{V}$.

Note that since $[\Gamma(2), \Gamma(2)]$ is a characteristic subgroup of $\Gamma(2)$, which in turn is normal in $\Sp(T_2(J))$, we have $[\Gamma(2), \Gamma(2)] \lhd \Sp(T_2(J))$.  In particular, $[\Gamma(2), \Gamma(2)]$ is a normal subgroup of the image $G_2 \subset \Sp(T_2(J))$ of $\rho_2$, so $K_{\infty}^{\ab}$ is Galois over $K$ and the action of $G_2$ on $\Gamma(2)$ by conjugation induces an action of $G_2$ on $\Gamma(2)^{\ab}$.  This induces an $\ff_2$-linear action of $G_2$ on $V$.  It follows from a straightforward calculation that this action sends an element $\sigma \in G_2$ to the automorphism of $V$ determined by $[i, j] \mapsto [\bar{\sigma}(i), \bar{\sigma}(j)]$ for $1 \leq i < j \leq d$, where the permuataion $\bar{\sigma}$ is the image of $\sigma$ in $G_2 / \Gamma(2) \cong \Gal(K_1 / K) = S_d$.  We therefore have an action of $S_d$ on $V$, which we denote by $\psi : S_d \to \Aut(V)$.

As $V$ is a vector space over $\ff_2$ of dimension $2g$, one candidate for the action $\psi$ is the well-known \textit{standard representation} of $S_d$ over $\ff_2$ (see \cite[\S2.2]{wagner1976faithful} for the construction of the standard representations over characteristic $2$ of dimension $d - 1$ if $d$ is odd and of dimension $d - 2$ if $d$ is even).  In our situation, this turns out to be the case.

\begin{lemma} \label{lemma standard rep}

The $2g$-dimensional representation $(V, \psi)$ of $S_d$ defined above is isomorphic to the standard representation.

\end{lemma}

\begin{proof}

We fix a symplectic ordered basis $\{a_1, ... , a_g, b_1, ... , b_g\}$ of $T_2(J)$, i.e. a basis satisfying $e_2(a_i, b_i) = -1 \in \zz_2$ for $1 \leq i \leq g$ and $e_2(a_i, b_j) = 0$ for $1 \leq i < j \leq g$ and such that the image of each $a_i$ (resp. each $b_i$) in $J[2]$ is represented by the even-cardinality subset of roots given by $\{\alpha_{2i - 1}, \alpha_{2i}\}$ (resp. $\{\alpha_{2i}, ... , \alpha_{2g + 1}\}$) (see the statement and proof of \cite[Corollary 2.11]{mumford1984tata} for a description of the elements in $J[2]$ in terms of the roots $\alpha_i$).  In the following argument, we use the description of $[\Gamma(2), \Gamma(2)]$ given by \cite[Proposition 2.1]{sato2010abelianization} as the subgroup of matrices (with respect to our symplectic basis) in $\Sp(T_2(J))$ which lie in $\Gamma(4)$ and whose $(i, i + g)$th and $(i + g, i)$th elements are divisible by $8$ for $1 \leq i \leq g$.  Note in particular that the automorphisms given by $v \mapsto v + 4e_2(v, a_i) a_i$ and $v \mapsto v + 4e_2(v, b_i) b_i$ each have $\pm 4$ as their $(i, i + g)$th or $(i + g, i)$th entries respectively, so their images in $\Gamma(4) / [\Gamma(2), \Gamma(2)]$ are nontrivial and distinct and in fact form a basis for the $\ff_2$-space $\Gamma(4) / [\Gamma(2), \Gamma(2)]$.

We claim that the representation $(V, \psi)$ is faithful.  Indeed, suppose that $\psi$ has nontrivial kernel.  Since $d \geq 5$, this implies that the kernel of $\psi$ contains $A_d \lhd S_d$.  Consider any element of $V$, written as a linear combination $\sum_{1 \leq i < j \leq 2g + 1} c_{i, j} [i, j] \in V$ with $c_{i, j} \in \ff_2$; by our assumptions this element must be fixed by $A_d$.  But then the $2$-transitivity of $A_d$ implies that the $c_{i, j}$'s are all equal, so that $V$ is spanned by the element $\sum_{1 \leq i < j \leq d} [i, j]$, which contradicts the fact that $V$ is $2g$-dimensional.

If $g \geq 3$, then \cite[Theorem 1.1]{wagner1976faithful} implies that $(V, \psi)$ is isomorphic to the standard representation, and we are done.  We therefore assume that $g = 2$.  It follows from the statement and proof of \cite[Lemma 3.2(iii),(iv)]{wagner1976faithful} that $(V, \psi)$ is isomorphic to the standard representation if and only if the transpositions in $S_d$ map to transvections in $\Aut(V)$ (in this context a transvection is defined to be any operator $A \in \Aut(V)$ such that $A - 1$ has rank $1$).  It therefore suffices to show that the transposition $(1 2) \in S_d$ acts on $V$ as a transvection.  This is equivalent to the claim that any element of $G_2$ whose image modulo $2$ is $(1 2) \in S_d = \Gal(K_1 / K)$ acts by conjugation on $\Gamma(4) / [\Gamma(2), \Gamma(2)] \cong \Gal(K_{\infty}^{\ab} / K_2)$ (which is the dual of $V$) as a transvection.  For any $a \in T_2(J)$, let $T_a \in \Aut(V)$ denote the automorphism given by $v \mapsto v + e_2(v, a) a$.  Then we see from the description of elements of $J[2]$ in terms of subsets of the set of $\alpha_i$'s which was mentioned above that the image of $T_{a_1} \in \Sp(T_2(J))$ modulo $2$ is $(1 2)$; since (as noted above) $\{T_{a_1}^4, T_{a_2}^4, T_{b_1}^4, T_{b_2}^4\}$ is an $\ff_2$-basis of $\Gamma(4) / [\Gamma(2), \Gamma(2)]$, we only need to calculate the conjugates of each of these basis elements by $T_{a_1}$.  We compute that $T_{a_1}$ commutes with each of them except for $T_{b_1}^4$, and that 
\begin{equation} T_{a_1} T_{b_1}^4 T_{a_1}^{-1} \equiv T_{b_1}^4 T_{a_1}^4 \ (\textrm{mod} \ [\Gamma(2), \Gamma(2)]). \end{equation}
Thus, we have seen that $T_{a_1}$ minus the identity operator acts on the $\ff_2$-space $\Gamma(4) / [\Gamma(2), \Gamma(2)]$ by sending all basis elements to $0$ except for $T_{b_1}^4$, which it sends to $T_{a_1}^4$.  Therefore, $T_{a_1}$ acts as a transvection, as desired.

\end{proof}

Now the following proposition suffices to prove Theorem \ref{thm main}(a) when $k = \cc$.

\begin{prop} \label{prop main g>1}

The subspace $V$ defined above is generated by the images in $\widetilde{V}$ of the elements $\prod_{j \neq i} \sqrt{\gamma_{i, j}} \in K_2^{\times}$ for $1 \leq i \leq 2g + 1$.

\end{prop}

\begin{proof}

For $1 \leq i < j \leq 2g + 1$, we write $[i, j]' = [j, i]'$ for the elements of $K_2^{\times} / (K_2^{\times})^2$ represented by $\sqrt{\gamma_{i, j}}$ (note that $[i, j]' = [i, j]$ in the $d = 2g + 1$ case); we need to show that $V$ is spanned as an $\ff_2$-space by the set $\{\sum_{j \neq i} [i, j]'\}_{1 \leq i \leq 2g + 1}$.

We know from Lemma \ref{lemma standard rep} that the action $\psi : S_d \to \Aut(V)$ defines the standard representation of $S_d$ of dimension $d - 1$ (resp. $d - 2$) if $d$ is odd (resp. even).  We first assume that $d = 2g + 1$.  By the construction of the standard representation, there exist elements $v_i \in V$ for $1 \leq i \leq 2g + 1$ which span $V$ and satisfy the unique linear relation $\sum_{1 \leq i \leq 2g + 1} v_i = 0$, and such that $S_{2g + 1}$ acts on the set of $v_i$'s by $v_i^{\sigma} = v_{\sigma(i)}$ for each permutation $\sigma$.  We claim that $v_1 = \sum_{j \neq 1} [1, j]$, from which it follows by acting on $v_1$ by any transposition $(1 i)$ that $v_i = \sum_{j \neq i} [i, j]$ for $2 \leq i \leq 2g + 1$, and we get the desired spanning set for $V$.

As $v_1$ is obviously nontrivial, some $[s, t]$ appears in its expansion as a linear combination of basis elements of $\widetilde{V}$.  Suppose that $1 \in \{s, t\}$.  Then the elements $[1, j]$ appear in the expansion of $v_i$ for all $2 \leq j \leq 2g + 1$, due to the fact that $v_1$ is fixed by every permutation in $S_{2g + 1}$ which fixes $1$.  If, on the other hand, $1 \notin \{s, t\}$, then by a similar argument, \textit{all} elements $[s, t]$ with $1 \notin \{s, t\}$ appear in the expansion of $v_1$.  It follows that either (i) $v_1 = \sum_{j \neq 1} [1, j]$, (ii) $v_1 = \sum_{s, t \neq 1} [s, t]$, or (iii) $v_1 = \sum_{1 \leq s < t \leq 2g + 1} [s, t]$.  In case (ii), we see that $\sum_{1 \leq i \leq 2g + 1} v_i = \sum_{1 \leq s < t \leq 2g + 1} [s, t] \neq 0$, a contradiction.  In case (iii), we see that $v_1$ is fixed by all elements of $S_{2g + 1}$ and therefore, all the elements $v_i$ are equal, which contradicts the fact that $V$ has dimension $2g$.  Therefore, (i) holds, and we are done.

Now assume that $d = 2g + 2$.  By the construction of the standard representation, there exist elements $v_{i, j} = v_{j, i} \in V$ for $1 \leq i < j \leq 2g + 2$ such that for any $i$, $\{v_{i, j}\}_{j \neq i}$ spans $V$ and satisfies the unique linear relation $\sum_{j \neq i} v_{i, j} = 0$; such that $v_{s, t} + v_{s, j} = v_{t, j}$ for distinct $s, t, j$; and such that $S_{2g + 2}$ acts on the set of $v_{i, j}$'s by $v_{i, j}^{\sigma} = v_{\sigma(i), \sigma(j)}$ for each permutation $\sigma$.  We claim that $v_{1, 2g + 2} = \sum_{2 \leq j \leq 2g + 1} [1, j]'$, from which we can again see by acting on $v_{1, 2g + 2}$ by any transposition $(1 i)$ that $v_{i, 2g + 2} = \sum_{j \neq i} [i, j]'$ for $2 \leq i \leq 2g + 1$, and we again get the desired spanning set for $V$.

By a similar analysis to what was done for the $d = 2g + 1$ case, we deduce that $v_{1, 2g + 2}$ is some linear combination of the elements $[i, 2g + 2]$, $\sum_{2 \leq j \leq 2g + 1} ([1, j] + [2g + 2, j])$, and $\sum_{2 \leq s < t \leq 2g + 1} [s, t]$.  We first note that $v_{1, 2g + 2}$ cannot be the sum of an odd number of terms $[s, t]$, because then that would be the case for each other $v_{i, 2g + 2}$, and then the $2g + 1$ terms $v_{i, 2g + 2}$ could not sum to $0$.  We also note that, as in the $d = 2g + 1$ case, the element $v_{1, 2g + 1}$ can neither be trivial nor equal to the sum $\sum_{1 \leq s < t \leq 2g + 2} [s, t]$.  Finally, it is straightforward to check that if $\sum_{2 \leq s < t \leq 2g + 1} [s, t]$ appears in the expansion of $v_{1, 2g + 2}$, then the $v_{s, t} + v_{s, j} = v_{t, j}$ property does not hold.  Our only remaining choice is that $v_{1, 2g + 2} = \sum_{2 \leq j \leq 2g + 1} ([1, j] + [2g + 2, j]) = \sum_{2 \leq j \leq 2g + 1} [1, j]'$, and we are done.

\end{proof}

\subsection{The $g = 1$ case}

Lemma \ref{lemma abelianization}, together with the results of \S\ref{sec2} and the fact that $K_{\infty} \subset K^{\unr}$, imply that $K_{\infty}^{\ab}$ is an extension of $K_2 = K_1(\sqrt{\gamma_{1, 2}}, \sqrt{\gamma_{1, 3}}, \sqrt{\gamma_{2, 3}})$ obtained by adjoining $2$ independent $4$th roots of products of the elements $\gamma_{i, j} \in K_1$ (recall from \cite{yelton2017note} that $K_1$ is generated over $K$ by the $\gamma_{i, j}$'s both when $d = 3$ and when $d = 4$).  Therefore, in this case, we get via Kummer theory a canonical identification of $\Gal(K_{\infty}^{\ab} / K_2)$ with some subgroup $V \subset K_2^{\times} / (K_2^{\times})^4$; the submodule $V$ of the $\zz / 4\zz$-module $K_2^{\times} / (K_2^{\times})^4$ is free of rank $2$.  In fact, $V$ is contained in the rank-$3$ free submodule $\widetilde{V} \subset K_2^{\times} / (K_2^{\times})^4$ generated by images of the elements $\sqrt{\gamma_{i, j}} \in K_2^{\times}$.  We denote each of these images by $[i, j]' = [j, i]' \in \widetilde{V}$.  We now proceed to explicitly determine the rank-$2$ submodule $V \subset \widetilde{V}$.

As with the $g \geq 2$ case, we have an action of $G_2$ on $V$ which factors through the quotient $\Gal(K_1 / K)$; this quotient is isomorphic to $S_3$ both when $d = 3$ and when $d = 4$.  This action sends a permutation $\bar{\sigma} \in S_3$ to the automorphism of $V$ determined by $[i, j]' = [\bar{\sigma}(i), \bar{\sigma}(j)]'$ for $1 \leq i < j \leq 3$.  We again write $\psi : S_3 \to \Aut(V)$ for this action.  The following proposition suffices to prove Theorem \ref{thm main}(b) when $k = \cc$.

\begin{prop} \label{prop main g=1}

The submodule $V \subset \widetilde{V}$ is generated by 

$$\{[1, 2]' + [1, 3]' + 2[2, 3]', [2, 3]' + [2, 1]' + 2[3, 1]', [3, 1]' + [3, 2]' + 2[1, 2]'\} \subset \widetilde{V}.$$

\end{prop}

\begin{proof}

We first note that the action $\psi$ has trivial kernel, by the same argument as was used in the proof of Proposition \ref{prop main g>1}.  Since $\Aut(V / 2V) \cong \SL_2(\ff_2) \cong S_3$, this implies that the induced action of $\psi$ modulo $2V$ is isomorphic to the standard representation of $S_3$.  By essentially the same argument that was used in the proof of Proposition \ref{prop main g>1} for the odd-degree case, this shows that $V / 2V$ is spanned by the images modulo $2$ of the elements $\sigma_{j \neq i} [i, j]'$ for $i = 1, 2, 3$, implying that 
$$2[1, 2]' + 2[1, 3]', \ 2[2, 3]' + 2[2, 1]', \ 2[3, 1]' + 2[3, 2]' \in V.$$
  Therefore, the element $2[1, 2]' + 2[1, 3]' + 2[2, 3]' \in \widetilde{V}$ cannot lie in $V$, because otherwise $V$ would contain the subgroup generated by $2[1, 2]', 2[1, 3]', 2[2, 3]'$, which is isomorphic to $(\zz / 2\zz)^3$, contradicting the fact that $V$ has rank $2$.

Let $\Phi : \widetilde{V} \to \zz / 4\zz$ be the functional sending $c_1[2, 3]' + c_2[3, 1]' + c_3[1, 2]'$ to $c_1 + c_2 + c_3 \in \zz / 4\zz$.  We claim that $V$ coincides with the kernel of $\Phi$.  Suppose that for some $v \in V$, we have $\Phi(v) \neq 0$.  Since $\Phi(\psi(\bar{\sigma})(v)) = \Phi(v)$ for all $\bar{\sigma} \in S_3$, we have $\Phi(\sum_{\bar{\sigma} \in A_3} \psi(\bar{\sigma})(v)) = 3\Phi(v) \neq 0$.  But $\sum_{\bar{\sigma} \in A_3} \psi(\bar{\sigma})(v) \in V$ is fixed by $A_3 \lhd S_3$ and is therefore some nontrivial multiple of $[1, 2]' + [1, 3]' + [2, 3]' \in \widetilde{V}$, which contradicts the fact that $2[1, 2]' + 2[1, 3]' + 2[2, 3]' \notin V$.  Since the image of $\Phi$ has order $4$, its kernel has order $64 / 4 = 16$, hence the claim.  Meanwhile, the generators given in the statement of the proposition lie in the kernel of $\Phi$, and the statement now follows from the elementary verification that the group they generate also has order $16$.

\end{proof}

\subsection{The action of $-1$}

We have shown that parts (a) and (b) of Theorem \ref{thm main} hold when $k = \cc$; we will now prove that the element $-1 \in \Gamma(2) \cong \Gal(K_{\infty} / K_1)$ acts as stated in Theorem \ref{thm main}(c), and then the full theorem will be proved over the complex numbers.  In the odd-degree case, it follows immediately from the fact that $\rho^{\tp}(\Sigma) = -1$ by Proposition \ref{prop rho top}(b) combined with Lemma \ref{lemma pure braid action abelian} that since $\gamma_{i, j} = \alpha_j - \alpha_i$, the element $-1$ acts trivially (resp. by sign change) on the generators of $K_{\infty}^{\ab} / K_2$ listed in the statement of the theorem if $g$ is even (resp. if $g$ is odd).  Similarly in the even-degree case, it follows immediately from the fact that $\rho^{\tp}(\Sigma') = -1$ by Proposition \ref{prop rho top}(c) combined with Lemma \ref{lemma pure braid action abelian} that since $\gamma_{i, j} = (\alpha_j - \alpha_i) \prod_{l \neq i, j} (\alpha_{2g + 2} - \alpha_l)$, the element $-1$ acts trivially (resp. by sign change) on the generators of $K_{\infty}^{\ab} / K_2$ listed in the statement of the theorem if $g$ is even (resp. if $g$ is odd).

\section{Proof of the theorem in the general case} \label{sec4}

We have shown that Theorem \ref{thm main} holds when $k = \cc$; we will now prove that this suffices to show that Theorem \ref{thm main} holds in general.  In this section, whenever the ground field $k$ is specified (e.g. $k = \qq$), we will write $K_k$ (e.g. $K_{\qq}$) for the extension of $k$ obtained by adjoining the symmetric functions of the independent transcendental elements $\alpha_1, ... , \alpha_d$; our goal is to show that the particular generators of $(K_{\cc})_{\infty}^{\ab} / K_{\cc}$ that we found in \S\ref{sec2} and \S\ref{sec3} can also be used to generate $K_{\infty}^{\ab} / K$ (and that any Galois element mapping to $-1 \in \GSp(T_2(J))$ acts in the same way on them), where $K = K_k$ for any field $k$ of characteristic different from $2$.  The rest of this section will be devoted to proving the following proposition.

\begin{prop} \label{prop descent}

To prove Theorem \ref{thm main}, it suffices to prove the statement when $k = \cc$.

\end{prop}

We assume that $g \geq 2$ and only prove part (a) of the theorem as well as part (c) for the $g \geq 2$ case, noting that the claims for the $g = 1$ case result from very similar arguments.  In what follows, we will freely use the obvious fact that given an abelian variety $A$ over a field $F$ and an extension $F' / F$, the finite algebraic extension $F'(A[2^n])$ coincides with the compositum $F' F(A[2^n])$ for any $n \geq 1$.  We first need the following lemmas.

\begin{lemma} \label{lemma full Galois image}

Let $k$ be any field of characteristic different from $2$ with separable closure $\bar{k}$.  Then we have 

a) $\Gal(K_n(\mu_2) / K_1(\mu_2)) \cong \Gamma(2) / \Gamma(2^n)$ for $n \geq 1$ and thus $\Gal(K_{\infty}(\mu_2) / K_1(\mu_2)) \cong \Gamma(2)$; and 

b) $K_n \cap \bar{k} = k(\zeta_{2^n})$ for $n \geq 1$ and thus $K_{\infty} \cap \bar{k} = k(\mu_2)$.

\end{lemma}

\begin{proof}

The author has shown (a) for $k$ of characteristic $0$ (as \cite[Proposition 4.1]{yelton2015images}) but the following argument proves (a) in the case of positive characteristic also. We first claim that the image of $\rho_2$ in $\GSp(T_2(J))$ contains a transvection given by $v \mapsto v + e_2(v, a)a$ for some $a \in T_2(J) \smallsetminus 2T_2(J)$. This follows from the discussion in \cite[\S2.3]{anni2017constructing} (see also \cite[\S3.v]{dokchitser2009regulator} and \cite[\S9-10]{grothendieck1972modeles} and note that the argument holds in positive characteristic as well).  Meanwhile, since the polynomial defining the hyperelliptic curve $C$ has full Galois group, the image of $\bar{\rho}_2$ is isomorphic to $S_{2g + 1}$ or $S_{2g + 2}$.  It now follows from \cite[Theorem 2.1.1, \S2.2]{brumer2017large} that the image $G_2$ of $\rho_2$ in $\GSp(T_2(J))$ contains $\Gamma(2) \lhd \Sp(T_2(J))$. After restricting to the absolute Galois group of $K_1(\mu_2)$, this image coincides with $\Gamma(2)$, and (a) immediately follows.

The linear disjointness of $K_n(\mu_2)$ and $\bar{k}K_1$ over $K_1(\mu_2)$ follows immediately from the fact that $\Gal(\bar{k} K_n / \bar{k} K_1) \cong \Gal(K_n(\mu_2) / K_1(\mu_2)) \cong \Gamma(2) / \Gamma(2^n)$ by part (a).  Moreover, it is clear from the well-known description of $2$-torsion fields discussed above that $K_1(\mu_2) \cap \bar{k} = k(\mu_2)$, so we get $K_n(\mu_2) \cap \bar{k} = (K_n(\mu_2) \cap \bar{k}K_1) \cap \bar{k} = k(\mu_2)$.  It follows that we have $K_{\infty} \cap \bar{k} = k(\mu_2)$, so to prove part (b) it suffices to show that $K_n \cap k(\mu_2) = k(\zeta_{2^n})$ for $n \geq 1$.

Let $\widetilde{\Gamma}(2^n) \lhd \GSp(T_2(J))$ denote the kernel of reduction modulo $2^n$ for each $n \geq 1$, and write $G_2 \subset \GSp(T_2(J))$ for the image of $\rho_2$.  It is clear that $G_2 \cap \widetilde{\Gamma}(2^n)$ is isomorphic to the subgroup of $\Gal(K_{\infty} / K)$ fixing $K_n$; meanwhile, part (a) says that $G_2 \cap \Gamma(2) = \Gamma(2)$ is isomorphic to the Galois subgroup fixing $k(\mu_2)$. Therefore, $K_n \cap k(\mu_2)$ is the fixed field of the subgroup of $\Gal(K_{\infty} / K)$ generated by $G_2 \cap \widetilde{\Gamma}(2^n)$ and $\Gamma(2)$, which is easily seen to coincide with the kernel of the mod-$2^n$ determinant map $G_2 \to (\zz / 2^n \zz)^{\times}$. But by equivariance of the Weil pairing, the fixed field of this subgroup coincides with $k(\zeta_{2^n})$, and we are done.

\end{proof}

Let $J_0$ denote the Jacobian of the hyperelliptic curve $C_0$ defined over $K_{\qq}$ given by the equation in (\ref{eq hyperelliptic}).  Note that $C_0$ admits a smooth model $\mathfrak{C}$ over 
$$S := \Spec(\zz[\textstyle\frac{1}{2}, \{\alpha_{i}\}_{1 \leq i \leq d}, \{(\alpha_{i} - \alpha_{j})^{-1}\}_{1 \leq i < j \leq 2g + 1}]^{S_d}),$$
 where the superscript ``$S_d$" indicates taking the subring of invariants under the obvious permutation action on the $\alpha_i$'s.
 Define $\mathfrak{J} \to S$ to be the abelian scheme representing the Picard functor of the scheme $\mathfrak{C} \to S$ (see \cite[Theorem 8.1]{milne1986jacobian}).  Note that the ring $\zz[\textstyle\frac{1}{2}, \{\alpha_{i}\}_{1 \leq i \leq d}, \{(\alpha_{i} - \alpha_{j})^{-1}\}_{1 \leq i < j \leq 2g + 1}]$, along with all subrings of invariants under finite groups of automorphisms, is integrally closed; in particular, $\mathcal{O}_S := \zz[\textstyle\frac{1}{2}, \{\alpha_{i}\}_{1 \leq i \leq d}, \{(\alpha_{i} - \alpha_{j})^{-1}\}_{1 \leq i < j \leq 2g + 1}]^{S_d}$ is integrally closed.

For each $n \geq 1$, \cite[Proposition 20.7]{milne1986abelian} implies that the kernel of the multiplication-by-$2^n$ map on $\mathfrak{J} \to S$, which we denote by $\mathfrak{J}[2^n] \to S$, is a finite \'{e}tale group scheme over $S$.  Since the morphism $\mathfrak{J}[2^n] \to S$ is finite, $\mathfrak{J}[2^n]$ is an affine scheme; we write $\mathcal{O}_{S, n} \supset \mathcal{O}_S$ for the minimal extension of scalars under which $\mathfrak{J}[2^n]$ becomes constant.  It follows from the fact that $\mathcal{O}_S$ is integrally closed and from the finite \'{e}taleness of $\mathfrak{J}[2^n]$ that $\mathcal{O}_{S, n}$ is also integrally closed; its fraction field coincides with $K_0(J_0[2^n])$.

Let $\mathcal{O}_{S, \infty}^{\ab} \supset \mathcal{O}_S$ denote the integrally closed extension whose fraction field coincides with the maximal abelian subextension $(K_{\qq(\mu_2)})_{\infty}^{\ab}$ of $(K_{\qq})_{\infty} / (K_{\qq})_1(\mu_2)$.  Lemmas \ref{lemma abelianization} and \ref{lemma full Galois image} together imply that $(K_{\qq})_2(\mu_2) \subsetneq (K_{\qq(\mu_2)})_{\infty}^{\ab} \subsetneq (K_{\qq})_3(\mu_2)$; that the extension $(K_{\qq(\mu_2)})_{\infty}^{\ab} / (K_{\qq})_1(\mu_2)$ has Galois group isomorphic to $(\zz / 2\zz)^{2g^2 - g} \times (\zz / 4\zz)^{2g}$; and that the analogous statements hold over each $\ff_p$.  Thus it is clear that for each prime $p \neq 2$, the fraction field of $\mathcal{O}_{S, \infty}^{\ab} / (p)$ coincides with the subfield of $(K_{\ff_p})_3(\mu_2)$ fixed by the kernel of the map $\bar{\pi} : \Gamma(2) / \Gamma(8) \twoheadrightarrow (\zz / 2\zz)^{2g^2 - g} \times (\zz / 4\zz)^{2g}$ induced by the abelianization map $\pi$ and is therefore the maximal abelian subextension of $(K_{\ff_p})_{\infty} / (K_{\ff_p})_1(\mu_2)$.  Moreover, if $k$ is any field of characteristic $p \neq 2$, then it similarly follows from Lemmas \ref{lemma abelianization} and \ref{lemma full Galois image} that $(K_{k(\mu_2)})_{\infty}^{\ab}$ coincides with the subfield of $K_3(\mu_2)$ fixed by the kernel of $\bar{\pi}$.

Note that $k(\mu_2)$ contains $\mathfrak{f}(\mu_2)$, where the prime subfield $\mathfrak{f}$ is $\qq$ (resp. $\ff_p$) if the characteristic of $k$ is $0$ (resp. $p \geq 3$).  It then follows from the linear disjointness of $K_3(\mu_2)$ and $\bar{\mathfrak{f}}$ over $\mathfrak{f}(\mu_2)$ given by Lemma \ref{lemma full Galois image}(b) that the subfield of $K_3(\mu_2)$ fixed by the kernel of $\bar{\pi}$ coincides with the compositum of $K_1(\mu_2)$ with the subfield of $(K_{\mathfrak{f}})_3(\mu_2)$ fixed by the kernel of $\bar{\pi}$.  The extension $(K_{k(\mu_2)})_{\infty}^{\ab}$ is therefore generated over $K_1(\mu_2)$ by the generators of $\mathcal{O}_{S, \infty}^{\ab}$ over $\mathcal{O}_{S, 1}[\mu_2]$ (resp. by the images of these generators modulo $(p)$) if $k$ has characteristic $0$ (resp. if $k$ has characteristic $p \geq 3$).

It remains to show that these generators are the same ones appearing in Theorem \ref{thm main}, for which we need another lemma.

\begin{lemma} \label{lemma compositum}

The fields $(K_{\cc})_n$ for $n \geq 1$, $(K_{\cc})_{\infty}$, and $(K_{\cc})_{\infty}^{\ab}$ coincide with the compositums of $\cc$ with $(K_{\qq(\mu_2)})_n$, $(K_{\qq(\mu_2)})_{\infty}$, and $(K_{\qq(\mu_2)})_{\infty}^{\ab}$ respectively.

\end{lemma}

\begin{proof}

For $n \geq 1$, let $\theta_n : \Gal((K_{\cc})_{\infty} / (K_{\cc})_n) \to \Gal((K_{\qq})_{\infty} / (K_{\qq})_n(\mu_2))$ be the map given by the composition of the obvious inclusion map with the obvious restriction map.  It is shown in the proof of \cite[Proposition 4.1]{yelton2015images} that each $\theta_n$ is an isomorphism (this can also be deduced from Lemma \ref{lemma full Galois image}(a)). Since $\theta_1$ and $\theta_3$ are isomorphisms, they induce an isomorphism $\Gamma(2) / \Gamma(8) \cong \Gal((K_{\cc})_3 / (K_{\cc})_1) \stackrel{\sim}{\to} \Gal((K_{\qq})_3(\mu_2) / (K_{\qq})_1(\mu_2))$, the image of whose restriction to $\Gal((K_{\cc})_3 / (K_{\cc})_{\infty}^{\ab})$ fixes the subfield $(K_{\qq(\mu_2)})_{\infty}^{\ab}$.  It follows from the definition of $\theta_3$ that $K_{\infty}^{\ab} = \cc (K_{\qq(\mu_2)})_{\infty}^{\ab}$.

\end{proof}

We now claim that $\{\gamma_{i, j}\}_{1 \leq i < j \leq 2g + 1}$ is a set of generators for $\mathcal{O}_{S, 2}[\zeta_8]$ over  $\mathcal{O}_{S,1}[\zeta_8]$.  Indeed, we see that $(K_{\cc})_1(\{\gamma_{i, j}\}_{1 \leq i < j \leq 2g + 1}) = (K_{\cc})_2 = \cc (K_{\qq})_2$ and 
\begin{equation} \Gal((K_{\qq})_2 / (K_{\qq})_1) \cong \widetilde{\Gamma}(2) / \widetilde{\Gamma}(4) = \Gamma(2) / \Gamma(4) \times \langle \iota \rangle \cong (\zz / 2\zz)^{2g^2 + g} \times \langle \iota \rangle, \end{equation}
 using Lemma \ref{lemma full Galois image}(a) and (b) and Lemma \ref{lemma abelianization}, where $\iota \in \Gal((K_{\qq})_2 / (K_{\qq})_1)$ is any automorphism that acts on $\qq(\zeta_4)$ by complex conjugation. It follows that $\mathcal{O}_{S, 2}$ is generated over $\mathcal{O}_{S, 1}$ by $\zeta_4$ and the square roots of integral elements $a_{i, j}\gamma_{i, j}$ for some $a_{i, j} \in \zz$ for $1 \leq i < j \leq 2g + 1$.  Since the extension $\mathcal{O}_{S, 2}$ is unramified over $\mathcal{O}_S$, we have $a_{i, j} \in \{\pm1, \pm2\}$.  But $\sqrt{\pm 1}, \sqrt{\pm 2} \in \zz[\zeta_8] \subset \mathcal{O}_{S, 1}[\zeta_8]$, so we have $\mathcal{O}_{S, 2}[\zeta_8] = \mathcal{O}_{S, 1}[\zeta_8, \{\sqrt{\gamma_{i, j}}\}_{1 \leq i < j \leq 2g + 1}]$, as claimed.

Next we find formulas for generators of $\mathcal{O}_{S, \infty}^{\ab}$ over $\mathcal{O}_{S, 1}(\mu_2)$ using the ones we have shown for $k = \cc$.  We know that $\mathcal{O}_{S, \infty}^{\ab} \supsetneq \mathcal{O}_{S, 2}(\mu_2) = \mathcal{O}_{S, 1}[\mu_2, \{\sqrt{\gamma_{i, j}}\}_{1 \leq i < j \leq 2g + 1}]$, and so, by Lemma \ref{lemma abelianization}, $\mathcal{O}_{S, \infty}^{\ab}$ is generated over $\mathcal{O}_{S, 2}(\mu_2)$ by square roots of $2g$ independent integral elements.  Then it is clear from Lemma \ref{lemma compositum} that we may choose these $2g$ elements to be of the form $a_i\sqrt{\scriptstyle\prod_{j \neq i} \gamma_{i, j}}$ for some $a_i \in \zz[\mu_2]$ for $1 \leq i \leq 2g$ and that the extension also contains a square root of $a_{2g + 1}\sqrt{\scriptstyle\prod_{j \neq 2g + 1} \gamma_{2g + 1, j}}$ for some $a_{2g + 1} \in \zz[\mu_2]$.  Using the fact that $\mathcal{O}_{S, \infty}^{\ab}$ is Galois over $\mathcal{O}_S[\mu_2]$, from conjugating by Galois automorphisms that fix $\qq(\mu_2)$ but permute the $\alpha_i$'s, we see that we may choose the elements $a_1, ... , a_{2g + 1}$ to be the same element $a \in \zz[\mu_2]$.  Note that the product of these $2g + 1$ elements $a\sqrt{\scriptstyle\prod_{j \neq i} \gamma_{i, j}}$ can be written as $\pm a^{2g + 1} \prod_{1 \leq i < j \leq 2g + 1} \gamma_{i, j}$, and this product must have a square root in $\mathcal{O}_{S, \infty}^{\ab}$.  But we already know that $\pm\prod_{1 \leq i < j \leq 2g + 1} \gamma_{i, j}$ has a square root in $\mathcal{O}_{S, 2}[\mu_2]$, so we have $\sqrt{a} \in \mathcal{O}_{S, \infty}^{\ab}$.  Since $\sqrt{a}$ is algebraic over $\qq(\mu_2)$, we get $\sqrt{a} \in \qq(\mu_2)$ by Lemma \ref{lemma full Galois image}(b).  Then $\mathcal{O}_{S, \infty}^{\ab}$ is generated over $\mathcal{O}_{S, 2}[\mu_2]$ by the elements $\sqrt{a}\sqrt[4]{\scriptstyle\prod_{j \neq i} \gamma_{i, j}}$.  Thus, the fraction field $(K_{\qq(\mu_2)})_{\infty}^{\ab}$ of $\mathcal{O}_{S, \infty}^{\ab}$ (resp. the fraction field $(K_{\ff_p(\mu_2)})_{\infty}^{\ab}$ of $\mathcal{O}_{S, \infty}^{\ab} / (p)$ for each prime $p \neq 2$) is generated over $(K_{\qq(\mu_2)})_2$ (resp. $(K_{\ff_p(\mu_2)})_2$) by the elements given in Theorem \ref{thm main}(a).

What we have shown above is that given any field $k$, the statement of Theorem \ref{thm main} holds over $k(\mu_2)$.  It is now clear that $K_1(\mu_2, \{\sqrt{\gamma_{i, j}}\}_{1 \leq i < j \leq 2g + 1}, \{\sqrt[4]{\scriptstyle\prod_{j \neq i} \gamma_{i, j}}\}_{1 \leq i \leq 2g + 1})$ is a subextension of $(K_{k(\mu_2)})_{\infty} = K_{\infty}(\mu_2) = K_{\infty} / K_1$.  If $\zeta_4 \in k$, then this subextension is clearly Kummer and therefore abelian, and it must be maximal abelian since there is no larger subextension which is abelian over $K_1(\mu_2)$.  If $\zeta_4 \notin k$, then $K_{\infty}^{\ab}$ must be the largest subextension of $K_1(\mu_2, \{\sqrt{\gamma_{i, j}}\}_{1 \leq i < j \leq 2g + 1}, \{\sqrt[4]{\scriptstyle\prod_{j \neq i} \gamma_{i, j}}\}_{1 \leq i \leq 2g + 1})$ which is abelian over $K_1$; this clearly coincides with $K_2(\mu_2) = K_1(\mu_2, \{\sqrt{\gamma_{i, j}}\}_{1 \leq i < j \leq 2g + 1})$.  Thus, the statement of Theorem \ref{thm main}(a) is proved over $k$.

Finally, let $\sigma \in G_K$ be an element such that $\rho_2(\sigma) = -1 \in \GSp(T_2(J))$.  Then it is clear from tracing through the above arguments that $\sigma$ acts on $\mathcal{O}_{S, \infty}^{\ab}$ by changing the signs of the generators $\sqrt{\gamma_{i, j}}$ and by fixing (resp. changing the signs of) the remaining generators $\scriptstyle \sqrt{\prod_{j \neq i} \gamma_{i, j}}$ if $g$ is even (resp. if $g$ is odd), and that therefore, it acts this way on $K_{\infty}^{\ab}$, proving part (c) over $k$.

\bibliographystyle{plain}
%\bibliography{bibfile}

\end{document}